\documentclass[12pt]{article}
\usepackage{amsmath}
\usepackage{graphicx}
\usepackage{natbib}
\usepackage{url} 

\usepackage{enumerate}
\usepackage{natbib}
\usepackage{url} 
\usepackage{graphicx, float}
\usepackage{epsf,psfrag}
\usepackage{latexsym}
\usepackage{amsmath, amsthm, amssymb, mathrsfs, amsfonts}
\usepackage{varioref}
\usepackage[table]{xcolor}
\usepackage{hhline}
\usepackage{epstopdf}
\usepackage{subcaption}
\usepackage[inline]{enumitem}
\usepackage{authblk}
\usepackage{subcaption}
\renewcommand{\theequation}{\arabic{equation}}
\newcommand{\beq}{\begin{equation}}
	\newcommand{\eeq}{\end{equation}}

\newtheorem{Theorem}{Theorem}
\newtheorem*{Theorem*}{Theorem}

\newtheorem{Lemma}{Lemma}
\newtheorem{Corollary}{Corollary}

\newtheorem{Definition}{Definition}
\newtheorem*{Definition*}{Definition}

\newtheorem{Example}{Example}
\newcommand\numberthis{\addtocounter{equation}{1}\tag{\theequation}}

\def\ba{\begin{array}}
	\def\ea{\end{array}}

\newcommand{\blind}{0}

\begin{document}

	\def\spacingset#1{\renewcommand{\baselinestretch}%
		{#1}\small\normalsize} \spacingset{1}

	
	\if0\blind
	{
		\title{\bf Asymptotic Normality for Plug-in Estimators of Generalized Shannon's Entropy}
		\author{Jialin Zhang\hspace{.2cm}\\
			Department of Mathematics and Statistics, Mississippi State University\\
			and \\
			Jingyi Shi \\
			Department of Mathematics and Statistics, Mississippi State University}
		\maketitle
	} \fi
	
	\if1\blind
	{
		\bigskip
		\bigskip
		\bigskip
		\begin{center}
			{\LARGE\bf Entropic Plots for Tail Type Classification}
		\end{center}
		\medskip
	} \fi
	
	\bigskip
	\begin{abstract}
		Shannon's entropy is one of the building blocks of information theory and an essential aspect of Machine Learning methods (e.g., Random Forests). Yet, it is only finitely defined for distributions with fast decaying tails on a countable alphabet. The unboundedness of Shannon's entropy over the general class of all distributions on an alphabet prevents its potential utility from being fully realized. To fill the void in the foundation of information theory, Zhang (2020) proposed generalized Shannon's entropy, which is finitely defined everywhere. The plug-in estimator, adopted in almost all entropy-based ML method packages, is one of the most popular approaches to estimating Shannon's entropy. The asymptotic distribution for Shannon's entropy's plug-in estimator was well studied in the existing literature. This paper studies the asymptotic properties for the plug-in estimator of generalized Shannon's entropy on countable alphabets. The developed asymptotic properties require no assumptions on the original distribution. The proposed asymptotic properties allow interval estimation and statistical tests with generalized Shannon's entropy.
	\end{abstract}
	
	\noindent%
	\vfill
	
	\newpage

\section{Introduction}

Shannon's entropy, introduced by \cite{shannon1948mathematical}, is one of the building blocks of Information Theory and a key aspect of Machine Learning (ML) methods (e.g., Random Forests). It is one of the most popular measurement on countable alphabet\footnote{An countable alphabet is a space that could be either finite, or countable infinite; the elements in an alphabet can be either ordinal ($e.g.$, numbers) or non-ordinal ($e.g.,$ letters).}, particularly on non-ordinal space with categorical data. For example, in \cite{li2017feature}, all reviewed feature selection methods on non-ordinal space boiled down to a function of Shannon's entropy. In addition, Shannon's entropy is one of the most important foundations for all tree-based ML algorithms, sometimes substitutable with the Gini impurity index \cite{banerjee2019tree} \cite{mienye2019prediction} \cite{hssina2014comparative}. As one of the essential information-theoretical quantities, Shannon's entropy and its estimation are widely studied in the past decades \cite{miller1954maximum} \cite{harris1975statistical} \cite{esty1983normal} \cite{paninski2003estimation} \cite{zhang2012entropy} \cite{zhang2012normal} \cite{zhang2013asymptotic}.

Nevertheless, Shannon's entropy is only finitely defined for distributions with fast decaying tails \cite{baccetti2013infinite}. It is never known if the real distribution yields a finite Shannon's entropy in practice. Furthermore, all existing results on Shannon’s entropy require it to be finitely defined, which results in a usage restriction when adopting the entropy-based methods. This is, in fact, a void in the foundation of all Shannon's entropy-related results. To address the deficiency of Shannon's entropy, \cite{zhang2020generalized} proposed generalized Shannon's entropy (GSE) and showed that GSE enjoys all utilities of a finite Shannon's entropy. In addition, GSE is finitely defined on all distributions. Due to the advantages of GSE and the deficiency of Shannon's entropy, the use of Shannon's entropy should eventually be transited to GSE. To aid the transition, the estimation of GSE needs to be studied. In practice, the plug-in estimator, adopted in almost all entropy-based ML method packages, is one of the most popular approaches to estimating Shannon’s entropy. For plug-in estimation of GSE, asymptotic properties are needed for statistical tests and confidence intervals. This paper aims to study the asymptotic properties of plug-in estimators of GSE.

The rest of this paper is organized as follows. Section 2 formally states the problem and gives our main results. In Section 3, we provide a small-scale simulation study. In Section 4, we discuss the potential of GSE. Proofs are postponed to Section 5.

\section{Main Results}

Let $Z$ be a random element on a countable alphabet $\mathscr{Z}=\left\{z_{k} ; k \geq 1\right\}$ with an associated distribution $\mathbf{p}=\left\{p_{k} ; k \geq 1\right\}$. Let the cardinality or support on $\mathscr{Z}$ be denoted $K=\sum_{k \geq 1} 1\left[p_{k}>0\right]$, where $1[\cdot]$ is the indicator function. $K$ is possibly finite or infinite. Let $\mathscr{P}$ denote the family of all distributions on $\mathscr{Z}$. Shannon's entropy, $H$, is defined as 
\beq
H=H(Z)=-\sum_{k \geq 1} p_{k} \ln p_{k}.
\eeq
To state our main result, we need to state Definition \ref{def-CDOTC} and \ref{def-GSE} given by \cite{zhang2020generalized}, and Definition \ref{def-pluginGSE}.

\begin{Definition}[Conditional Distribution of Total Collision (CDOTC)]
	Given $\mathscr{Z}=\left\{z_{k} ; k \geq 1\right\}$ and $\mathbf{p}=\left\{p_{k}\right\}$, consider the experiment of drawing an identically and independently distributed (iid) sample of size $m$ ($m\geq 2)$. Let $C_{m}$ denote the event that all observations of the sample take on a same letter in $\mathscr{Z}$, and let $C_{m}$ be referred to as the event of total collision. The conditional probability, given $C_{m}$, that the total collision occurs at letter $z_{k}$ is
	$$
	p_{m, k}=\frac{p_{k}^{m}}{\sum_{i \geq 1} p_{i}^{m}},
	$$
	where $m \geq 2$. $\mathbf{p}_{m}=\left\{p_{m, k}\right\}$ is defined as the $m$-th order CDOTC.
	\label{def-CDOTC}
\end{Definition}

\begin{Definition}[Generalized Shannon's Entropy (GSE)]
	Given $\mathscr{Z}=\left\{z_{k} ; k \geq 1\right\}$, $\mathbf{p}=\left\{p_{k}\right\}$, and $\mathbf{p}_{m} = \{ p_{m, k} \}$, generalized Shannon's entropy (GSE) is defined as
	$$
	H_m (Z) = - \sum_{k\geq 1} p_{m, k} \ln p_{m, k},
	$$
	where $p_{m,k}$ is defined in Definition \ref{def-CDOTC}, and $m = 2, 3, \dots$ is the order of GSE. GSE with order $m$ is referred to as the $m$-th order GSE.
	\label{def-GSE}
\end{Definition}

It is clear that $\mathbf{p}_{m}$ is a probability distribution induced from $\mathbf{p}=\left\{p_{k}\right\}$. To help understand Definition \ref{def-CDOTC} and \ref{def-GSE}, Example \ref{example-CDOTC} and \ref{example-GSE} are provided as follows.

\begin{Example}[The 2nd order CDOTC]
	Given $\mathscr{Z}=\left\{z_{k} ; k \geq 1\right\}$ and $\mathbf{p}=\left\{p_{k}\right\}=\{ 6 k^{-2} / \pi ^ 2; k = 1, 2, 3, \dots \}$, the 2nd order CDOTC is then defined as
	$$
	\mathbf{p}_{2} = \{p_{2, k}\},
	$$
	where
	$$
	p_{2,k}=\frac{p_{k}^{2}}{\sum_{i \geq 1} p_{i}^{2}} = \frac{36 k^{-4}/\pi ^ 4}{\sum_{i \geq 1} \left[ 36 i^{-4} / \pi ^4\right]} = \frac{k^{-4}}{\sum_{i \geq 1} i^{-4} }
	$$
	for $k = 1, 2, 3, \dots$.
	\label{example-CDOTC}
\end{Example}

\begin{Example}[The 2nd order GSE]
	Given $\mathscr{Z}=\left\{z_{k} ; k \geq 1\right\}$, $\mathbf{p}=\left\{p_{k}\right\}=\{ 6 k^{-2} / \pi ^ 2; k = 1, 2, 3, \dots \}$, and $\mathbf{p}_{2}= \{p_{2, k}\} = \{ \frac{k^{-4}}{\sum_{i \geq 1} i^{-4} }; k = 1, 2, \dots \}$, the 2nd order GSE, $H_2 (Z)$, is then defined as
	$$
	H_2 (Z) = - \sum_{k\geq 1} p_{2, k} \ln p_{2, k},
	$$
	where $p_{2, k}$ is given in Example \ref{example-CDOTC}.
	\label{example-GSE}
\end{Example}

The definition of the plug-in estimator of GSE is stated in Definition \ref{def-pluginGSE}.

\begin{Definition}[Plug-in estimator of GSE]
	Let $Z_1, Z_2, \dots, Z_n$ be independent and identically distributed (iid) random variables taking values in $\mathscr{Z}=\left\{z_{k} ; k \geq 1\right\}$ with distribution $\mathbf{p}$. For each $k = 1, 2, \dots$, let $Y_k = \sum_{i=1}^{n} 1{[Z_i = z_k]}$ be the sample count of observations in category $z_k$, and let $\hat{p}_{k} = Y_k / n$ be the sample proportion. The plug-in estimator for the $m$-th order GSE, $\hat{H}_m(Z)$, is defined as
	\begin{align*}
		\hat{H}_m(Z) = & - \sum_{k\geq 1} \left[ \hat{p}_{m, k} \ln \hat{p}_{m, k} \right] \\
		= & - \sum_{k\geq 1} \left[ \frac{\hat{p}_{k}^{m}}{\sum_{i \geq 1} \hat{p}_{i}^{m}} \ln \frac{\hat{p}_{k}^{m}}{\sum_{i \geq 1} \hat{p}_{i}^{m}} \right].
	\end{align*}
	\label{def-pluginGSE}
\end{Definition}

Our main results are stated in Theorem \ref{thm-GSEplug-in}, Corollary \ref{cor-GSEplug-in}, and Corollary \ref{cor-GSEplug-in-finite}.

\begin{Theorem} Let $\mathbf{p} = \{p_k\}$ be a probability distribution on a countably infinite alphabet, without any further conditions,
	$$
	\sqrt{n} \left({\hat{H}_m (Z) - H_m (Z)}\right) \xrightarrow{d} N(0,\sigma_m ^2),
	$$
	where
	$$
	{\sigma}_m ^2= \sum_{k=1}^{\infty} \left[\frac{m^2}{{p}_k}\left({p}_{m,k} \ln {p}_{m,k} + {p}_{m,k}  {H}_m(Z) \right)\right]^2.
	$$
	\label{thm-GSEplug-in}
\end{Theorem}

\begin{Corollary} Let $\mathbf{p} = \{p_k\}$ be a probability distribution on a countably infinite alphabet, without any further conditions,
	$$
	\sqrt{n} \left(\frac{\hat{H}_m (Z) - H_m (Z)}{\hat{\sigma}_m}\right) \xrightarrow{d} N(0,1),
	$$
	where
	\beq
	\hat{\sigma}_m ^2= \sum_{k=1}^{\infty} \left[\frac{m^2}{\hat{p}_k}\left(\hat{p}_{m,k} \ln \hat{p}_{m,k} + \hat{p}_{m,k}  \hat{H}_m(Z) \right)\right]^2. \numberthis  \label{sigmahat}
	\eeq
	\label{cor-GSEplug-in}
\end{Corollary}

\begin{Corollary}
	Let $\mathbf{p} = \{p_k; k=1, 2, \dots, K\}$ be a non-uniform probability distribution on a countably finite alphabet, without any further conditions,
	$$
	\sqrt{n} \left(\frac{\hat{H}_m (Z) - H_m (Z)}{\hat{\sigma}_m}\right) \xrightarrow{d} N(0,1),
	$$
	where
	$$
	\hat{\sigma}_m ^2= \sum_{k=1}^{K} \left[\frac{m^2}{\hat{p}_k}\left(\hat{p}_{m,k} \ln \hat{p}_{m,k} + \hat{p}_{m,k}  \hat{H}_m(Z) \right)\right]^2.
	$$
	\label{cor-GSEplug-in-finite}
\end{Corollary}

Corollary \ref{cor-GSEplug-in-finite} is a special case of Theorem \ref{thm-GSEplug-in}. All proofs are provided in Section 5.

\section{Simulations}

One of the main applications of our results is the ability to construct confidence intervals, and hence testing hypothesis. Specifically, Corollary \ref{cor-GSEplug-in} implies that an asymptotic $(1-\alpha)100\%$ confidence interval for $H_m$ is given by
\beq
\left(\hat{H}_m - z_{\alpha /2} \frac{\hat{\sigma}_m}{\sqrt{n}}, \hat{H}_m + z_{\alpha /2} \frac{\hat{\sigma}_m}{\sqrt{n}}\right), \numberthis \label{confidence_interval}
\eeq
where $\hat{\sigma}_m$ is given by (\ref{sigmahat}) and $z_{\alpha / 2}$ is a number such that $P\left(Z>z_{\alpha / 2}\right)=\alpha / 2$ and $Z \sim$ $N(0,1)$. In this section, we give a small scale simulation study to check the finite sample performance of this confidence interval.

We consider Zeta distribution
\beq
P(x=k) = \frac{1}{\zeta(s)} k^{-s}, \quad k \in  \{1, 2, \dots\}
\eeq
with $s=1.5$, where $\zeta(s)$ is the Riemann zeta function given by
\beq
\zeta(s) = \sum_{n=1}^{\infty} \frac{1}{n^s}.
\eeq
We set $s=1.5$ because such Zeta distribution has asymptotic normality with $\hat{H}_m$ but does not have asymptotic normality with $\hat{H}$ (\cite{zhang2012normal}).

The simulations were performed as follows. For the given distribution, we obtained a random sample of size $n$ and used it to evaluate a $95 \%$ confidence interval for a given index using (\ref{confidence_interval}). We then checked to see if the true value of the $H_m$ was in the interval or not. This was repeated 5000 times, and the proportion of times when the true value was in the interval was calculated. When the asymptotics works well, this proportion should be close to $0.95$. We repeated this for sample sizes ranging from 10 to 1000 in increments of 10. The results for $s=1.5$, order $m=2$ are given in Figure \ref{fig_order2}, and the results for $s=1.5$, order $m=3$ are given in Figure \ref{fig_order3}.

The results suggest that convergence is fast, particularly when the order $m=2$. We conjecture that this may be caused by the fact that, when $m$ is larger, the probabilities in the corresponding CDOTC are smaller and hence require a larger sample size for convergence. Although GSE with order $m\geq3$ may shed some light on specific information, GSE with order $m=2$ is enough to well exist with asymptotic properties for any valid
underlying probability distribution $\mathbf{p}$.

\section{Discussion}

The proposed asymptotic properties in Corollary \ref{cor-GSEplug-in} and \ref{cor-GSEplug-in-finite} make it possible for interval estimation and statistical tests. Based on the simulation results, the convergence is quite fast, particularly under order $m=2$. Note that a GSE with order $m=2$ already enjoys all asymptotic properties without any assumption on original distribution $\mathbf{p}$.

We recommend using GSE with order $m=2$ in place of Shannon's entropy in all entropy-based methods. The proposed asymptotic results also allow interval estimation and statistical tests on the modified entropy-based methods that replaced Shannon's entropy with GSE. By replacing Shannon's entropy with GSE, one still enjoys all the benefits of Shannon's entropy with a pretty fast convergence speed. Moreover, using GSE is risk-free compared to Shannon's entropy because Shannon's entropy 1) does not exist on some thick-tailed distributions and 2) requires thinner tail distribution for some asymptotic properties.

To further unlock the utility of GSE, future research is needed on the Generalized Mutual Information (GMI), also proposed in \cite{zhang2020generalized}. The proposed asymptotic properties in this article directly provide asymptotic normality for the plug-in estimator of GMI when the real underlying GMI is not 0. The asymptotic behavior for the plug-in estimator of GMI when the real underlying GMI is 0 remains an open question, which we will address in future work.

\section{Proofs}

The proofs require several lemmas. The first lemma is state below.
\begin{Lemma}[\cite{zhang2012normal} and \cite{grabchak2018asymptotic}]
	Assume that $\sum_{k=1}^{\infty} p_{k}\left|\log p_{k}\right|^{2}<\infty$ and that there is a deterministic sequence $K(n)$ with $K(n) \rightarrow \infty$ such that $\lim _{n \rightarrow \infty} K(n) / \sqrt{n} \rightarrow 0$ and
	$$
	\lim _{n \rightarrow \infty} \sqrt{n} \sum_{k=K(n)}^{\infty} p_{k} \log p_{k}=0.
	$$
	In this case
	$$
	\sqrt{n}\left(\hat{H}_{n}-H\right) \stackrel{d}{\rightarrow} N\left(0, \sigma^{2}\right),
	$$
	where
	$$
	\sigma^{2}=\sum_{k=1}^{\infty} p_{k}\left(\log p_{k}\right)^{2}-\left(\sum_{k=1}^{\infty} p_{k} \log p_{k}\right)^{2}.
	$$
	Furthermore, if $\sigma>0$,
	$$
	\sqrt{n}\left(\frac{\hat{H}_{n}-H}{\hat{\sigma}_{n}}\right) \stackrel{d}{\rightarrow} N(0,1)
	$$
	where
	$$
	\hat{\sigma}_{n}^{2}=\sum_{k=1}^{\infty} \hat{p}_{k}\left(\log \hat{p}_{k}\right)^{2}-\left(\sum_{k=1}^{\infty} \hat{p}_{k} \log \hat{p}_{k}\right)^{2}.
	$$
	\label{lemma-hhat}
\end{Lemma}

Different proofs of Lemma \ref{lemma-hhat} are provided in  \cite{zhang2012normal} and \cite{grabchak2018asymptotic}.

The spirit for proof of Theorem \ref{thm-GSEplug-in} is to regard CDOTC as an original distribution and utilize the result from Lemma \ref{lemma-hhat}. Toward that end, several lemmas are needed and stated below.
\begin{Lemma}[Equivalent conditions in Lemma \ref{lemma-hhat}]
	For any valid distribution $\mathbf{p}$, let the corresponding CDOTC with order $m$ be denoted as $\mathbf{p}_m$, then
	$$
	\sum_{k=1}^{\infty} p_{m, k}\left|\log p_{m, k}\right|^{2}<\infty
	$$
	and that there is a deterministic sequence $K(n)$ with $K(n) \rightarrow \infty$ such that $\lim _{n \rightarrow \infty} K(n) / \sqrt{n} \rightarrow 0$ and
	$$
	\lim _{n \rightarrow \infty} \sqrt{n} \sum_{k=K(n)}^{\infty} p_{m, k} \log p_{m, k}=0.
	$$
	\label{lemma-conditions}
\end{Lemma}
\begin{Lemma}[$\sigma^2_m$ in Theorem \ref{thm-GSEplug-in}]
	In Theorem \ref{thm-GSEplug-in},
	$$
	{\sigma}_m ^2= \sum_{k=1}^{\infty} \left[\frac{m^2}{{p}_k}\left({p}_{m,k} \ln {p}_{m,k} + {p}_{m,k}  {H}_m(Z) \right)\right]^2.
	$$
	\label{lemma-sigma}
\end{Lemma}
\begin{Lemma}[$\hat{\sigma}^2 _m$ in Corollary \ref{cor-GSEplug-in}]
	In Corollary \ref{cor-GSEplug-in},
	$$
	\hat{\sigma}_m ^2= \sum_{k=1}^{\infty} \left[\frac{m^2}{\hat{p}_k}\left(\hat{p}_{m,k} \ln \hat{p}_{m,k} + \hat{p}_{m,k}  \hat{H}_m(Z) \right)\right]^2.
	$$
	\label{lemma-sigmahat}
\end{Lemma}
\begin{proof}[Proof of Lemma \ref{lemma-conditions}]
	Note that for any $\mathbf{p}$ to be a valid distribution, the tail of $\mathbf{p}$ must be thicker than $1/(k \ln k)$ because $\sum_{k\geq2} 1/(k \ln k)$ diverges. Hence $\mathbf{p}_m$ is thicker than $1/(k^2 \ln ^2 k)$ for any $m\geq 2$ by definition. It is shown in Example 3 of \cite{zhang2012normal} that such tail satisfies the mentioned conditions.
\end{proof}
\begin{proof}[Proof of Lemma \ref{lemma-sigma}]
Because of Lemma \ref{lemma-conditions}, $\sigma^2$ can be obtained under finite $K$ and then let $K\rightarrow \infty$. For a finite $K$, 	it can be verified that for $i = 1, \dots, K-1$,
$$
\frac{\partial H_m}{\partial p_i} = \left(\ln p_{m,K} - \ln p_{m, i}\right)\frac{mp_{m,i}}{p_i} - m\left(\frac{p_{m,i}}{p_{i}}-\frac{p_{m,K}}{p_K}\right)\left(H_m + \ln p_{m,K}\right).
$$
Let
$$
\begin{aligned}
	&v=\left(p_{1}, \cdots, p_{K-1}\right)^{\tau}, \\
	&\hat{v}=\left(\hat{p}_{1}, \cdots, \hat{p}_{K-1}\right)^{\tau} .
\end{aligned}
$$
We have $\sqrt{n}(\hat{v}-v) \stackrel{L}{\rightarrow} M V N(0, \Sigma(v))$, where $\Sigma(v)$ is the $(K-1) \times(K-1)$ covariance matrix given by
$$
\Sigma(v)=\left(\begin{array}{cccc}
	p_{1}\left(1-p_{1}\right) & -p_{1} p_{2} & \cdots & -p_{1} p_{K-1} \\
	-p_{2} p_{1} & p_{2}\left(1-p_{2}\right) & \cdots & -p_{2} p_{K-1} \\
	\cdots & \cdots & \cdots & \cdots \\
	-p_{K-1} p_{1} & -p_{K-1} p_{2} & \cdots & p_{K-1}\left(1-p_{K-1}\right)
\end{array}\right)
$$
According to the first-order Delta method,
$$
\sigma^2 _K = \nabla H_m^{T} \Sigma \nabla H_m =   \sum_{k=1}^{K} \left[\frac{m^2}{{p}_k}\left({p}_{m,k} \ln {p}_{m,k} + {p}_{m,k}  {H}_m(Z) \right)\right]^2.
$$
Given Lemma \ref{lemma-conditions}, let $K\rightarrow\infty$,
$$
\sigma^2 =  \sum_{k=1}^{\infty} \left[\frac{m^2}{{p}_k}\left({p}_{m,k} \ln {p}_{m,k} + {p}_{m,k}  {H}_m(Z) \right)\right]^2.
$$
\end{proof}
\begin{proof}[Proof of Lemma \ref{lemma-sigmahat}]
	Lemma \ref{lemma-sigmahat} is because of $ \hat{\sigma}_m^2 \stackrel{p}{\rightarrow} \sigma_m^2$. 
\end{proof}
\begin{proof}[Proof of Theorem \ref{thm-GSEplug-in} and Corollary \ref{cor-GSEplug-in}]
	With Lemma \ref{lemma-hhat}, \ref{lemma-conditions}, \ref{lemma-sigma}, \ref{lemma-sigmahat}, and Slutsky's theorem, Theorem \ref{thm-GSEplug-in} and Corollary \ref{cor-GSEplug-in} are proved. 
\end{proof}
\begin{proof}[Proof of Corollary \ref{cor-GSEplug-in-finite}]
	Corollary \ref{cor-GSEplug-in-finite} is a directly result of Theorem \ref{thm-GSEplug-in}, except under uniform distribution when $\nabla H_m = 0$ for all $m \geq 2$. 
\end{proof}

\section*{Figures}
\begin{figure}[H]
	\begin{center}
		\includegraphics[height=8cm, width=8cm]{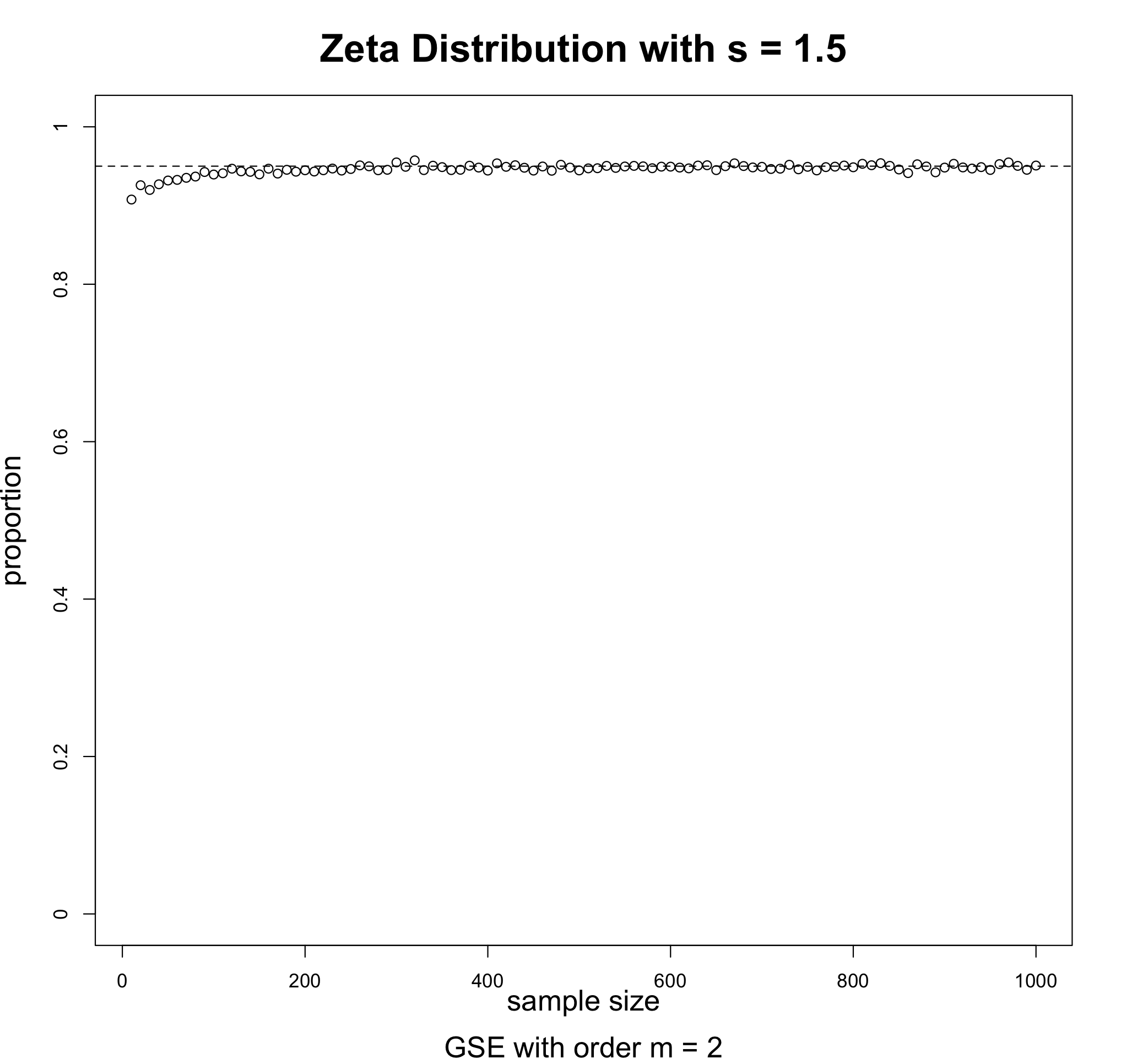}
	\end{center}
	\caption{Effectiveness of the 95\% confidence intervals as a function of sample size. Simulations from Zeta distribution with $s = 1.5$ and GSE with order $m=2$. The horizontal dashed line is at 0.95.}
	\label{fig_order2}
\end{figure}

\begin{figure}[H]
	\begin{center}
		\includegraphics[height=8cm, width=8cm]{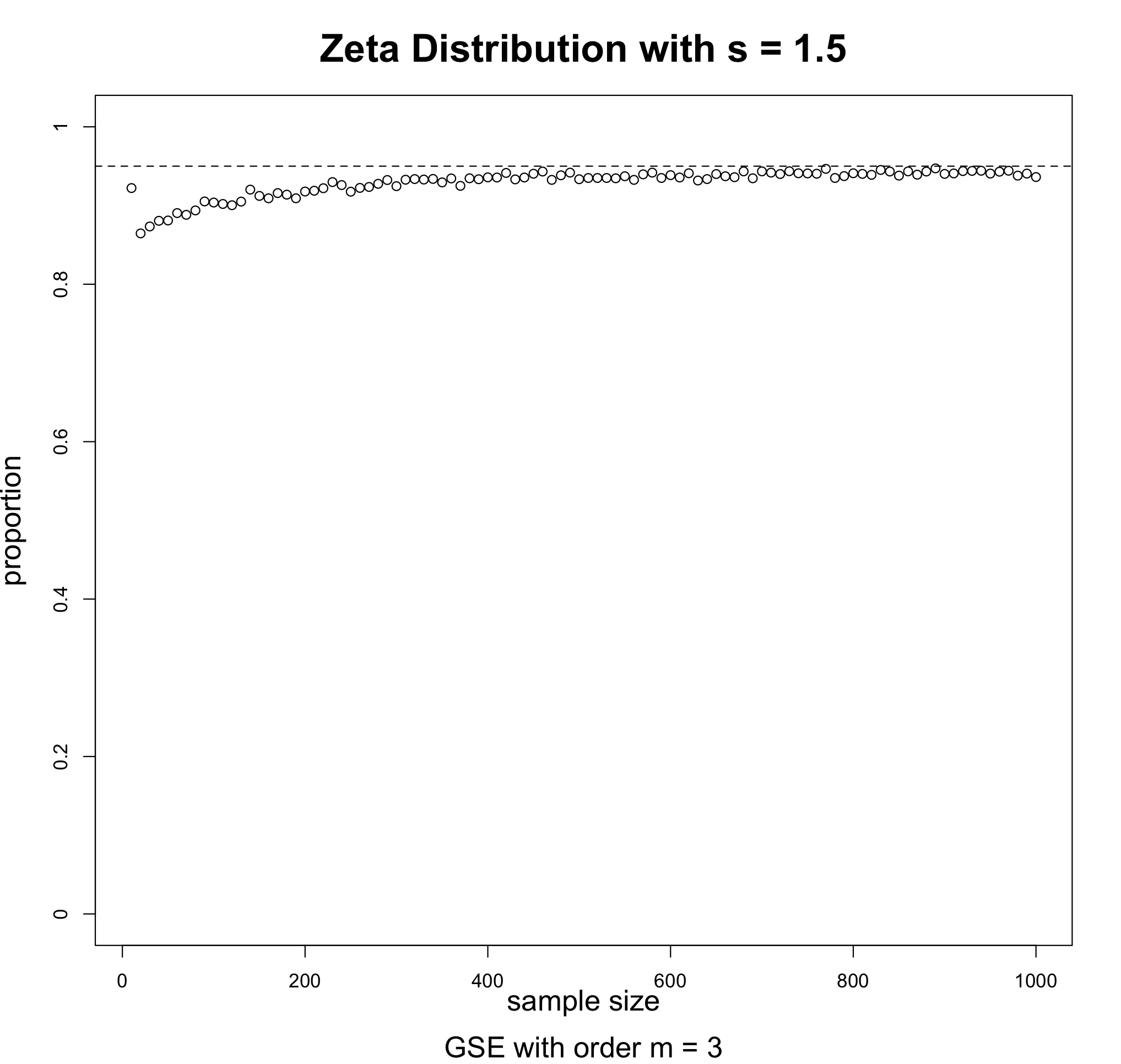}
	\end{center}
	\caption{Effectiveness of the 95\% confidence intervals as a function of sample size. Simulations from Zeta distribution with $s = 1.5$ and GSE with order $m=3$. The horizontal dashed line is at 0.95.}
	\label{fig_order3}
\end{figure}



\vspace{6pt}

\appendix
\bibliography{references_db}

\end{document}